\documentclass[a4paper,12pt]{amsart}
\usepackage{graphicx}
\usepackage{amsmath}
\usepackage{latexsym,lscape}
\usepackage{amsthm,float,autograph}
\usepackage{epic,latexsym,bezier,caption,amsbsy,psfrag,color,enumerate,amsfonts,amsmath,amscd,amssymb,comment}
\usepackage{epsfig}
\usepackage{rotating}
\usepackage{graphics}
\usepackage[latin1]{inputenc}
\usepackage{multicol}
\usepackage{epstopdf}

\newtheorem{defi}{Definition}[section]
\newtheorem{theorem}[defi]{Theorem}
\newtheorem{lemma}[defi]{Lemma}

\newtheorem{corollary}[defi]{Corollary}
\newtheorem{question}{Question}
\theoremstyle{definition}
\newtheorem{problem}{Problem}[section]
\newtheorem{example}[defi]{Example}

\newtheorem{remark}[defi]{Remark}

\newcommand{\notimplies}{%
\mathrel{{\ooalign{\hidewidth$\not\phantom{=}$\hidewidth\cr$\implies$}}}}

\begin{document}


\title{Distance-balanced graphs and Travelling Salesman Problems}

\author{Matteo Cavaleri}
\address{Matteo Cavaleri, Universit\`{a} degli Studi Niccol\`{o} Cusano - Via Don Carlo Gnocchi, 3 00166 Roma, Italia}
\email{matteo.cavaleri@unicusano.it}

\author{Alfredo Donno}
\address{Alfredo Donno, Universit\`{a} degli Studi Niccol\`{o} Cusano - Via Don Carlo Gnocchi, 3 00166 Roma, Italia}
\email{alfredo.donno@unicusano.it \textrm{(Corresponding author)}}

\begin{abstract}
For every probability $p\in[0,1]$ we define a distance-based graph property, the $p$TS-distance-balancedness, that in the case $p=0$ coincides with the standard distance-balancedness, and in the case $p=1$ is related to the Hamiltonian-connectedness. In analogy with the classical case, where the distance-balancedness of a graph is equivalent to the property of being self-median, we characterize the class of $p$TS-distance-balanced graphs in terms of their equity with respect to certain probabilistic centrality measures, inspired by the Travelling Salesman Problem. We prove that it is possible to detect this property looking at the classical distance-balancedness (and therefore looking at the classical centrality problems) of a suitable graph composition, namely the wreath product of graphs. More precisely, we characterize the distance-balancedness of a wreath product of two graphs in terms of the $p$TS-distance-balancedness of the factors.
\end{abstract}

\maketitle

\begin{center}
\end{center}

\section{Introduction}
The investigation of \emph{distance-balanced graphs} began in \cite{Handa}, though an explicit definition was provided later in \cite{distancebalancedintro,stronglyintro}. Such graphs generated a certain degree of interest also by virtue of their connection with centrality measures \cite{equalopportunity,scozzari} and with some well known and largely studied distance-based invariants, such as the Wiener and the Szeged index \cite{equalopportunity,dist,distancebalancedintro,khalifeh}. For instance, it was proven in \cite{dist} that, in the bipartite case, distance-balanced graphs maximize the Szeged index.\\
\indent Throughout the paper we will denote by $G=(V_G,E_G)$ a simple connected finite graph $G$ with vertex set $V_G$ and edge set $E_G$. We say that such a graph has order $n$ if $|V_G|=n$. For a pair of adjacent vertices $u,v\in V_G$ (we say $u\sim v$ in $G$) we define
\begin{eqnarray}\label{Wuv}
W_{uv}^G = \{z\in V_G\ | \ d_G(z,u) < d_G(z,v)\},
\end{eqnarray}
where $d_G(u,v)$ denotes the geodesic distance in $G$. In other words, $W_{uv}^G$ is the set of vertices of $G$ which are closer to $u$ than to $v$.

\begin{defi}\label{distdefi}
A graph $G=(V_G,E_G)$ is \emph{distance-balanced} if $|W_{uv}^G|=|W_{vu}^G|$, for every pair of adjacent vertices $u,v\in V_G$.
\end{defi}

Cyclic graphs and complete graphs are simple examples of distance-balanced graphs. More generally, it is known that the class of distance-balanced graphs contains vertex-transitive graphs \cite{stronglyintro}, which are graphs $G=(V_G,E_G)$ whose automorphism group $\operatorname{Aut}(G)$ acts transitively on the vertex set. On the other hand, the Handa graph $H_{24}$, introduced in \cite{Handa}, is an example of a non-vertex-transitive distance-balanced graph.\\
\indent Recall that semisymmetric graphs are regular graphs which are edge-transitive (the group $\operatorname{Aut}(G)$ acts transitively on the edge set) but not vertex-transitive graphs. In particular, a semisymmetric graph is bipartite, and the two sets of the bipartition coincide with the orbits of $\operatorname{Aut}(G)$. As for such graphs there exist no automorphism switching two adjacent vertices, they appear as good candidates to be not distance-balanced. However, in \cite{stronglyintro} it is explicitly proven that there exist infinitely many semisymmetric graphs which are not distance-balanced, as well as infinitely many semisymmetric graphs which are distance-balanced.\\
\indent In \cite{distancebalancedintro}, the behaviour of the four classical graph compositions with respect to the distance-balanced property is investigated. More precisely, it is shown that the Cartesian product $G\Box H$ of two connected graphs is distance-balanced if and only if both $G$ and $H$ are distance-balanced; the lexicographic product $G \circ H$ of two connected graphs is distance-balanced if and only if $G$ is distance-balanced and $H$ is regular; on the other hand, it is shown there, by explicit counterexamples, that the direct product $G\times H$ and the strong product $G \boxtimes H$ do not preserve the property of being distance-balanced.

In \cite{strongly}, in order to construct an algorithm that recognizes whether a given graph is distance-balanced or not, the authors establish a connection with some \emph{graph centrality measures}; more precisely, they characterize the distance-balancedness of a graph in terms of its median vertices, and therefore in terms of their total distance (also known as transmission in the literature).\\
\indent We denote the \emph{normalized total distance} of a vertex $u\in V_G$ as
\begin{equation*}
d_G(u):= \frac{1}{|V_G|} \sum_{v\in V_G}d_G(u,v),
\end{equation*}
which is the average of the distances of $u$ from each vertex of $G$.
A vertex $u\in V_G$ is said to be \emph{median} if $d_G(u)=\min_{v\in V_G} \{d_G(v)\}$. The graph $G$ is said to be \emph{self-median} if every vertex $u\in V_G$ is median.

\begin{theorem}\cite[Theorem 3.1]{strongly}\label{median}
A graph $G=(V_G,E_G)$ is distance-balanced if and only if it is self-median.
\end{theorem}

According to Theorem \ref{median}, distance-balanced graphs are graphs where all vertices have the same relevance in some sense, but they are not necessarily indistinguishable (notice that there are even examples of distance-balanced graphs with a trivial automorphism group \cite{knor}). Therefore, distance-balanced graphs are of special interest in the study of social networks, as all people in such graphs are \emph{equal} with respect to the total distance. A related measure of this equality is given by the \emph{opportunity index} \cite{equalopportunity}: distance-balanced graphs are characterized as those graphs whose opportunity index is zero.

In the present paper, aimed at generalizing distance-balancedness in a precise direction, we start exactly from this point of view, and we interpret the set of median vertices of a graph, and the whole class of distance-balanced graphs itself, as solutions of particular \emph{facility location} problems, very typical in graph centrality investigations.
In order to deeper understand this correspondence, let us suppose that $G=(V_G,E_G)$ represents a city; its vertex set is the set of the buildings/locations, the edges are connections between the buildings and then, for any $u,v\in V_G$, the geodesic distance $d_G(u,v)$ represents the distance between buildings $u$ and $v$, or the \emph{cost} of reaching the vertex $v$ from the vertex $u$. In these terms, the quantity $d_G(u)$ is the average distance of the location $u$ from all locations, and the median vertices are those vertices solving the following problem.
\begin{problem}\label{p1}
Find the location for a facility in order to minimize its average distance from all the buildings of the city.
\end{problem}

\noindent Consequently, distance-balanced graphs are those graphs whose vertices are all equal with respect to Problem \ref{p1}. That is, they solve this second problem.
\begin{problem}\label{p2}
Find a city where Problem \ref{p1} is solved by any location.
\end{problem}

From another point of view, that we will develop in the last part of the paper, our work can be interpreted as the investigation of the distance-balancedness in a \emph{wreath product of graphs}. The wreath product of graphs represents the graph analogue of the classical wreath product of groups, as it is true that the wreath product of the Cayley graphs of two finite groups is the Cayley graph of the wreath product of the groups. In \cite{gc}, this correspondence is proven in the more general context of generalized wreath products of graphs, inspired by the construction introduced in \cite{bayleygeneralized} for permutation groups. Also, observe that in \cite{coronanoi} the wreath product of matrices has been defined, in order to describe the adjacency matrix of the wreath product of two graphs: spectral computations using this matrix representation have been developed for some infinite families of wreath products in \cite{BCD,cocktail,compcomp}.

The paper is organized as follows. In Section \ref{secTS}, we consider two  optimization problems, namely Problem \ref{tp1} and Problem \ref{tp2}, that are the analogue, respectively, of Problem \ref{p1} and Problem \ref{p2}, where the centrality measure at the vertex $u$ is not yet the normalized total distance, but the quantity $d_G^p(u)$, that is, the expectation of the length of a shortest path starting from $u$ that satisfies some random requirements depending on the probability $p$. In particular, these new problems collapse to the classical ones in the case $p=0$.\\
\indent Problems \ref{tp1} and \ref{tp2} are of some interest on their own, given their connection with the Travelling Salesman Problem, which is among the most studied optimization problems, largely investigated in literature also in its several probabilistic versions (see, for instance, \cite{PTSP}).\\
\indent Then we extend the classical definition of distance-balanced graph by introducing the notion of \emph{$p$TS-distance-balanced} graph in Definition \ref{Tdb}, and we prove in Theorem \ref{Tmedia} a $p$TS analogue of Theorem \ref{median}: $p$TS-distance-balanced graphs are exactly the graphs that solve Problem \ref{tp2} (that is, the TS-version of Problem \ref{p2}). We present examples and non-examples of $p$TS-distance-balanced graphs.\\
\indent In Section \ref{wre}, we recall the definition of the wreath product $G\wr H$ of two graphs $G$ and $H$ (Definition \ref{defierschler}). It turns out that, when the order of $H$ is $m$, the uniform probability distribution on the vertices of $G\wr H$ is compatible, in a precise sense explained in Lemma \ref{ww}, with the model introduced in Section \ref{secTS} for $G$, when $p=\frac{m-1}{m}$.\\
\indent It follows that the TS-problems considered on the graph  $G$ are equivalent to the classical problems on the wreath product $G\wr H$, for a suitable choice of the graph $H$. More precisely we characterize, in Theorem \ref{teom}, the distance-balancedness of a wreath product in terms of $p$TS-distance-balancedness of its factors. Finally, we investigate the class of graphs that are $p$TS-distance-balanced for every $p\in[0,1]$, giving several equivalent characterizations in Theorem \ref{fine}. We conclude the paper by asking if this class actually coincides with the class of vertex-transitive graphs (Question \ref{qq}).


\section{$p$TS centrality}\label{secTS}
As a natural generalization of Problem \ref{p1}, suppose that every day each building (vertex) of the city (graph) $G=(V_G,E_G)$, independently, with the same probability $p\in [0,1]$, requires a visit from the facility and with probability $1-p$ does not. An example could be a postoffice with a postman delivering parcels. This setting is similar to that of the \emph{homogeneous probabilistic travelling salesman problem} \cite{gamba}, but here we want to find a location for the postoffice in order to minimize the expectation of the length of a shortest walk starting from the postoffice, visiting at least once each building waiting for a parcel, and finally arriving at the postman's house, that can be on each vertex with the same probability $\frac{1}{n}$ (observe that the postoffice and the postman's house locations may coincide). This set-up is justified if, for example, we have to decide the postoffice location prior to be aware of the location of the postman's house, or for example if every day the postman can be different. We are going to formalize this model in what follows.

\begin{defi}\label{roa}
Let $G=(V_G,E_G)$ be a graph and let $A\subseteq V_G$. We define a map $\rho_A$ on $V_G \times V_G$ such that, for any pair of vertices $u$ and $v$ in $V_G$, the number $\rho_A(u,v)$ is the length of a shortest walk joining  $u$ and $v$, visiting at least once all vertices in $A$.
\end{defi}

\begin{remark}\label{obs}
Let $G=(V_G,E_G)$ be a graph of order $n$, and let $u,v,z\in V_G$ and $A\subseteq V_G$.
We list some properties of the map $\rho_A$; see \cite{cavadonno} for more details.
\begin{itemize}
\item $\rho_{\emptyset}(u,v)=d_G(u,v)$
\item $\rho_A(u,v)=\rho_A(v,u)$ (Symmetry)
\item $\rho_{A\cup B}(u,v)\leq \rho_{A}(u,z)+ \rho_{B}(z,v)$ (Triangle inequality)
\item $B \subseteq A \implies \rho_B(u,v)\leq \rho_A(u,v)$ (Monotonicity)
\item $\rho_A(u,v)<n^2$
\end{itemize}
and combining the first with the third property we have
\begin{equation}\label{rhopiud}
|\rho_{A}(u,z)- \rho_{A}(v,z)|\leq d_G(u,v).
\end{equation}
\end{remark}

Let $G=(V_G,E_G)$  be a graph of order $n$, and let $u,v \in V_G$. A \emph{Hamiltonian path} from $u$ to $v$ in $G$ is a path from $u$ to $v$ visiting each vertex of $G$ exactly once. A \emph{Hamiltonian cycle}  is a Hamiltonian path which is a cycle. A graph is \emph{Hamiltonian} if it admits a Hamiltonian cycle, that is equivalent to say that  $\rho_{V_G}(u,u)=n$ for some, or equivalently, for every $u\in V_G$. A graph $G$ is \emph{Hamilton-connected} if, for every pair $u,v\in V_G$, there exists a Hamiltonian path from $u$ to $v$. It is easy to observe that
$$
\forall u,v\in V_{G},\;\rho_{V_G}(u,v)=
\left\{ \begin{array}{ll}
    n-1 & \hbox{if } u\neq v \\
    n & \hbox{if } u=v
  \end{array}
\right. \iff G \mbox{ is Hamilton-connected.}
$$
The computation of $\rho_{V_G}$ for Hamilton-connected graphs is rather easy; however, to determine $\rho_A$ is in general very hard. This is not the case for the easiest example of Hamilton-connected graph, that is, the complete graph $K_n$.

\begin{example}
Let $K_n = (V_{K_n},E_{K_n})$ be the \emph{complete graph on $n$ vertices}. For every nonempty $A\subseteq V_{K_n}$ and every $u,v\in V_{K_n}$ we have
$$
\rho_A(u,v)=\begin{cases}
|A|+1\; &\mbox{ if } u,v\notin A\\
|A|\;  &\mbox{ if } u\notin A, v\in A \mbox{ or viceversa} \\
|A|-1\;  &\mbox{ if } u,v\in A , u\neq v\\
|A|\;  &\mbox{ if } u=v\in A, |A|>1\\
0 \;  &\mbox{ if } u=v\in A, |A|=1.\\
\end{cases}
$$
\end{example}

The hypothesis that each vertex independently with probability $p$ requires a visit implies that the probability for a given subset $A\subseteq V_G$ to be the random subset waiting for the parcels is
\begin{equation}\label{pA}
p_A:=p^{|A|}(1-p)^{n-|A|}.
\end{equation}

Then we define the quantity $d_G^p(u)$, that is, the expected length of a walk from $u$, visiting the random set $A$ and arriving to the random vertex $v$ (uniformly distributed on $V_G$), as follows:
\begin{equation}\label{dp}
\begin{split}
d_G^p(u):=\frac{1}{n} \sum_{v\in V_G } \sum_{A\subseteq V_G} p_A \; \rho_A(u,v).
\end{split}
\end{equation}
\begin{remark}\label{caso0}${}$\\
If $p=0$ we have $p_A= \begin{cases} 1 \mbox{ if } A=\emptyset \\0 \mbox{ otherwise} \end{cases}$  and $d_G^p(u)=d_G(u)$.\\
If $p=\frac{1}{2}$ we have $p_A=\frac{1}{2^n}$ and ${d_G^p(u)=\frac{1}{2^nn} \sum_{v\in V_G} \sum_{A\subseteq V_G} \rho_A(u,v)}$.\\
If  $p=1$ we have $p_A= \begin{cases} 1 \mbox{ if } A=V_G \\0 \mbox{ otherwise}\end{cases}$ and ${d_G^p(u)=\frac{1}{n} \sum_{v\in V_G }\rho_{V_G}(u,v)}$.
\end{remark}

We are now in position to formulate the $p$TS versions of Problem \ref{p1} and Problem \ref{p2}, respectively.

\begin{problem}\label{tp1}
Find a vertex $u\in V_G$ such that $d_G^p(u)=\min_{v\in V_G} \{d_G^p(v)\}$.
\end{problem}

\begin{problem}\label{tp2}
Find a graph such that Problem \ref{tp1} is solved by any vertex.
\end{problem}

This leads us to introduce a notion of \emph{medianity} in this setting, as a solution of the above mentioned problems.

\begin{defi}
In a graph  $G=(V_G,E_G)$ a vertex $u\in V_G$ is  \emph{$p$TS-median}  if it solves Problem \ref{tp1}.  The graph $G$ is \emph{self-$p$TS-median} if it solves Problem \ref{tp2}.
\end{defi}

\begin{remark}
Notice that, as a consequence of Remark \ref{caso0}, when $p=0$ the Problem \ref{p1} and Problem \ref{p2} and their $p$TS versions, Problem \ref{tp1} and Problem \ref{tp2} respectively, are equivalent.
\end{remark}

In analogy with Equation \eqref{Wuv}, for any subset $A\subseteq V_G$ and any pair of adjacent vertices $u,v\in V_G$, we define the vertex subsets
$$
W_{uv}^A := \{z\in V_G\ | \ \rho_A(z,u) < \rho_A(z,v)\},
$$
and the expectation of their cardinality is
\begin{equation}\label{TWuv}
 w^p_{uv}:= \sum_{A\subseteq V_G}  p_A |W_{uv}^A|.
\end{equation}
A natural generalization of the distance-balancedness is given in the following definition.

\begin{defi}\label{Tdb}
A graph $G=(V_G,E_G)$ is \emph{$p$TS-distance-balanced} if $w_{uv}^p=w_{vu}^p$, for every pair of adjacent vertices $u,v\in V_G$.
A graph $G$ is \emph{TS-distance-balanced} if it is $p$TS-distance-balanced for each $p \in [0,1]$.
\end{defi}

The following is the  $p$TS-version of Theorem \ref{median}.
\begin{theorem}\label{Tmedia}
A graph  $G=(V_G,E_G)$ is $p$TS-distance-balanced if and only if it is self-$p$TS-median.
\end{theorem}
\begin{proof}
Observe that, as the graph $G=(V_G,E_G)$ is connected, the statement is proved if we show that, for every pair of adjacent vertices $u,v\in V_G$, one has:
\begin{eqnarray}\label{pluto}
d_G^p(u)-d_G^p(v)=\frac{1}{n}(w^p_{vu}-w^p_{uv}).
\end{eqnarray}
Now, by the definition of $d_G^p(u)$ in Equation \eqref{dp}, we get
\begin{align*}
d_G^p(u)-d_G^p(v)&=\frac{1}{n} \sum_{z\in V_G } \sum_{A\subseteq V_G}  p_A (\rho_A(u,z)-\rho_A(v,z))\\
&=  \frac{1}{n} \sum_{A\subseteq V_G}  p_A  \sum_{z\in V_G }(\rho_A(u,z)-\rho_A(v,z))= \frac{1}{n} \sum_{A\subseteq V_G}  p_A  (|W_{vu}^A|-|W_{uv}^A|)
\end{align*}
since, being $u$ and $v$ adjacent, by virtue of Equation \eqref{rhopiud}, we have
$$
\rho_A(u,z)-\rho_A(v,z)=\begin{cases}
1 \quad \mbox{ if } z\in\, W_{vu}^A &\\
-1 \quad \mbox{ if } z\in\, W_{uv}^A &\\
0 \quad \mbox{ otherwise. } &\\
\end{cases}
$$
Finally, by Equation \eqref{TWuv}, we have
$$
d_G^p(u)-d_G^p(v)= \frac{1}{n} \sum_{A\subseteq V_G}  p_A  (|W_{vu}^A|-|W_{uv}^A|)=\frac{1}{n} (w^p_{vu}-w^p_{uv}),
$$
that proves Equation \eqref{pluto}.
\end{proof}

\begin{remark}\label{obs2}
As we have already observed,
$$
G \mbox{ is distance-balanced} \iff G \mbox{ is {\scriptsize $0$}TS-distance-balanced}.
$$
Moreover, if, for two given vertices $u,v\in V_G$, there exists $\varphi\in \operatorname{Aut}(G)$ such that $\varphi(u)=v$, one has $d_G^p(u)=d_G^p(v)$, for every $p\in[0,1]$. This implies
$$
G \mbox{ is vertex-transitive} \implies G \mbox{ is TS-distance-balanced}.
$$
On the other hand, when $p=1$, Hamilton-connected graphs satisfy $d_G^1(u)=n-1+\frac{1}{n}$ by Remark \ref{caso0} for every vertex $u\in V_G$, and then:
$$
G \mbox{ is Hamilton-connected} \implies G \mbox{ is {\scriptsize $1$}TS-distance-balanced.}
$$
Notice that the converse of the last implication is not true, since there exist graphs which are vertex-transitive but not Hamilton-connected (e.g., cyclic graphs).
\end{remark}

\begin{example}\label{ewheel}
The graph $W_7$ depicted in Figure \ref{figwheel7} is the \emph{Wheel graph} on $7$ vertices. Being Hamilton-connected, the graph $W_7$ is {\scriptsize $1$}TS-distance-balanced. Clearly, it is not distance-balanced and then it is not {\scriptsize $0$}TS-distance-balanced.
\\By an explicit computation (brute-force) we computed $d^{1/2}_{W_7}(u)=\frac{1}{7\cdot 2^7} \cdot 3842$, whether $u$ is the central vertex or it belongs to the external cycle. As a consequence, the graph $W_7$ is {\scriptsize $\frac{1}{2}$}TS-distance-balanced. We found quite surprising that this graph, that has a so recognizable central vertex, presents such an equality property.
\begin{figure}[h]
\begin{center}
\includegraphics[width=0.35\textwidth]{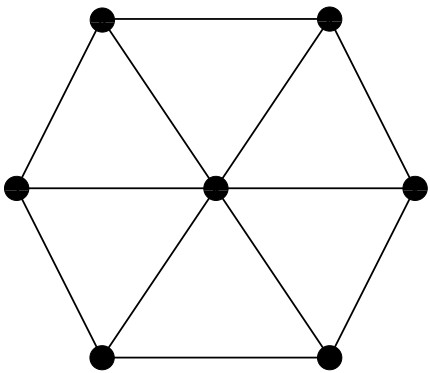}
\end{center}\caption{The Wheel graph $W_7$ on $7$ vertices.} \label{figwheel7}
\end{figure}
\end{example}

\begin{example}\label{epiccolo}
Let $H_9$ be the graph on $9$ vertices depicted in Figure \ref{figstrangegraph}. This graph has been introduced in \cite{dist} as the smallest example of a non-regular distance-balanced graph. In particular, it is {\scriptsize $0$}TS-distance-balanced, but an explicit computation gives $d^{1/2}_{H_9}(v_1)= \frac{1}{9\cdot2^9}\cdot 26688$ and $d^{1/2}_{H_9}(v_2)=\frac{1}{9\cdot2^9}\cdot 26656$. By virtue of Theorem \ref{Tmedia}, it is not {\scriptsize $\frac{1}{2}$}TS-distance-balanced.
\begin{figure}[h]
\begin{center}
\psfrag{v1}{$v_1$}\psfrag{v2}{$v_2$}
\includegraphics[width=0.35\textwidth]{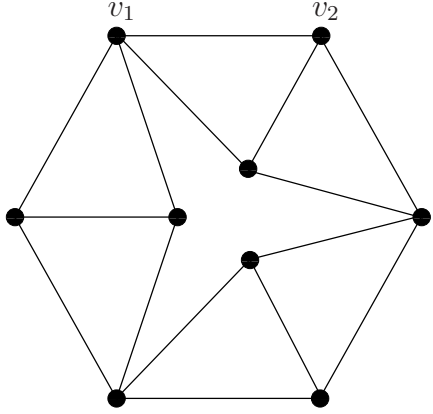}
\end{center}\caption{The graph $H_9$.} \label{figstrangegraph}
\end{figure}
\end{example}

\section{$p$TS-distance-balancedness and wreath product of graphs}\label{wre}

We start this section by recalling the definition of the wreath product of two graphs.

\begin{defi}\label{defierschler}
Let $G=(V_G, E_G)$ and $H=(V_H,E_H)$ be two graphs. Let us fix an enumeration of the vertices of $G$ so that $V_G=\{x_1,x_2,\ldots,x_{n}\}$. The \textit{wreath product} $G\wr
H$ is the graph with vertex set $V_H^{V_G}\times V_G=\{(y_1,\ldots, y_n)x_i \ | \ x_i \in V_G,  y_j\in V_H, j=1,\ldots, n \}$, where two vertices $u=(y_1,\ldots, y_{n})x_i$ and $v=(y'_1,\ldots, y'_{n})x_k$ are connected by an edge if:
\begin{itemize}
\item (edges of type I) either $i=k$ and $y_j=y'_j$ for every $j\neq i$, and $y_{i}\sim y'_{i}$ in $H$;
\item (edges of type II) or $y_j=y'_j$, for every $j=1,\dots, n$, and $x_i\sim x_k$ in $G$.
\end{itemize}
\end{defi}
It follows that if $|V_G|=n$ and $|V_H|=m$, the graph $G\wr H$ has $nm^n$ vertices. It is proved that  $G\wr H$ is connected, bipartite or vertex-transitive, if and only if both $G$ and $H$ are, respectively, connected, bipartite or vertex-transitive \cite{cavadonno}. Moreover, if $G$ is a regular graph of degree $r_G$ and $H$ is a regular graph of degree $r_H$, then the graph $G\wr H$ is an $(r_G+r_H)$-regular graph.

There exists a practical and clarifying interpretation of the graph wreath product, given by the \emph{Lamplighter random walk} \cite{woesssmall}.
Suppose that a lamplighter moves along $G$, so that each vertex of $G$ represents a possible \emph{position} of the lamplighter: at each vertex of $G$, there is a \emph{lamp}. The vertices of the graph $H$ represent the possible \emph{colors} of each lamp. Therefore, a vertex $(y_1,\ldots, y_{n})x_i$ in $G\wr H$ can be regarded as a \emph{configuration} of colors $(y_1,\ldots, y_{n})$ (each $y_j$ is from $V_H$) together with a particular position $x_i$ (from $V_G$) of the lamplighter: the lamp at the position $x_j\in V_G$ has the color $y_j \in V_H$. Two vertices of $G\wr H$ are connected if either they share the same configuration of colors and have adjacent positions for the lamplighter in $G$ (such an edge of type II corresponds to the situation of the lamplighter moving to a neighbor vertex in $G$ but leaving all lamps unchanged); or they share the same position but the configuration of colors differ, by two adjacent colors, exactly for the lamp associated at that position (such an edge of type I corresponds to the situation of the lamplighter changing the color of the lamp, to an adjacent color in $H$, in the position where he stays). For this reason, the wreath product $G\wr H$ is also called the Lamplighter graph, with base graph $G$ and color graph $H$.

The Lamplighter interpretation allows us to highlight the relationship between the wreath product of graphs and Problems \ref{tp1} and \ref{tp2}. We explicit this connection in the two following lemmas, where the distance and the normalized total distance in $G\wr H$ are expressed in terms of distance and normalized total distance in $H$, and of their $TS$-version in $G$.

\begin{lemma}[\cite{cavadonno}]\label{dista}
Let $G=(V_G,E_G)$ and $H=(V_H,E_H)$ be two graphs of order $n$ and $m$, respectively. For any vertices $u=(y_1,\ldots,y_n)x$, $v=(y_1',\ldots, y_n')x'\in V_{G\wr H}$, we have:
\begin{eqnarray*}
d_{G\wr H}(u,v)=\sum_{i=1}^n d_{H}(y_i,y_i')+\rho_{\delta(u,v)}(x,x'),
\end{eqnarray*}
where $\delta(u,v):= \{x_i\in V_G: y_i\neq y'_i \}.$
\end{lemma}

Observe that the sum $\sum_{i=1}^n d_{H}(y_i,y_i')$ can be interpreted as the geodesic distance in the $n$-th iterated Cartesian product of $H$ with itself. Moreover, in the Lamplighter interpretation, the subset $\delta(u,v)$ consists of those vertices $x_i$ of the base graph $G$ where the color of the associated lamps $y_i$ and $y'_i$ does not coincide. In other words, $\delta(u,v)$ is the set of the vertices of $G$ that the lamplighter has to visit in order to move from the lamp configuration of $u$ to that of $v$.

Notice that fixing a vertex $u$ in $G\wr H$ and considering a random vertex $v=(y_1',\ldots, y_n')x'$, with uniform probability $\frac{1}{n m^n}$ in $V_{G\wr H}$, induces a random subset $\delta(u,v)$ of $V_G$ and a random vertex $x'$ in $V_G$. The probability of the random subset $\delta(u,v)$ coincides with that given by Equation \eqref{pA} when $p=\frac{m-1}{m}$. This means that the model considered in Section \ref{secTS} can be simply derived by assigning a uniform probability distribution to the vertices of the wreath product. This is formally proved in Lemma \ref{ww}.

\begin{lemma}\label{ww}
Let $G=(V_G,E_G)$ and $H=(V_H,E_H)$ be two graphs of order $n$ and $m$, respectively. Let $u=(y_1,\ldots y_n) x\in V_{G\wr H}$ and $p=\frac{m-1}{m}$. Then:
\begin{equation*}\label{wiwr}
d_{G\wr H}(u)= \sum_{i=1}^n  d_H(y_i) + d^p_G(x).
\end{equation*}
\end{lemma}
\begin{proof}
From Theorem 17 in \cite{scozzari}, where the focus is on the non-normalized total distance, we have:
$$
n m^n d_{G\wr H}(u)= n m^{n-1}\sum_{i=1}^n  m d_H(y_i) + \sum_{A\subseteq V_G} (m-1)^{|A|} \sum_{x'\in V_G} \rho_A(x,x').
$$
Taking into account that $p=\frac{m-1}{m}$, we obtain:
\begin{eqnarray*}
d_{G\wr H}(u)&=& \sum_{i=1}^n  d_H(y_i) + \frac{1}{n}  \sum_{A\subseteq V_G} \sum_{x'\in V_G} \frac{(m-1)^{|A|} }{m^n}  \rho_A(x,x')\\
&= &\sum_{i=1}^n  d_H(y_i) + \frac{1}{n}  \sum_{A\subseteq V_G} \sum_{x'\in V_G} p^{|A|} (1-p)^{n-|A|}  \rho_A(x,x').
\end{eqnarray*}
The claim follows by using Equations \eqref{pA} and \eqref{dp}.
\end{proof}


In the following theorem we give necessary and sufficient conditions  for a vertex of a wreath product to be median and for the wreath product itself to be distance-balanced.

\begin{theorem}\label{teom}
Let $G=(V_G,E_G)$ and $H=(V_H,E_H)$ be two graphs with $|V_G|=n$, $|V_H|=m$. Let $u=(y_1,\ldots, y_n) x\in V_{G\wr H}$, and $p=\frac{m-1}{m}$. Then $u$ is median in $G\wr H$ if and only if $y_1,\dots,y_n$ are median in $H$  and  $x$ is $p$TS-median in $G$.
In particular:
\begin{equation*}
\begin{split}
G& \mbox{ is $p$TS-distance-balanced }\\
&H  \mbox{ is distance-balanced }
\end{split}
 \iff G\wr H \mbox{ is distance-balanced }.
\end{equation*}
\end{theorem}
\begin{proof}
Suppose that $y_{i^*}$ is not median in $H$, so that there exists $\bar{y}\in V_H$ such that $d_H(\bar{y})<d_H(y_{i^*})$. Denoting by $u'$ the vertex $(y_1,\ldots ,y_{i^*-1}, \bar{y},y_{i^*+1},\ldots, y_n)x\in V_{G\wr H}$, by Lemma \ref{ww} we have
$d_{G\wr H}(u')<d_{G\wr H}(u)$, and then $u$ is not median in $G\wr H$.\\
Similarly, suppose that $x$ is not $p$TS-median in $G$, so that there exists $\bar{x}\in V_G$ such that $d^p_G(\bar{x})<d^p_G(x)$. Denoting by $u''$ the vertex $(y_1,\ldots, y_n) \bar{x}\in V_{G\wr H}$, by Lemma \ref{ww} again we have $d_{G\wr H}(u'')<d_{G\wr H}(u)$, and then $u$ is not median in $G\wr H$.\\
Viceversa, suppose that $u$ is not median in $G\wr H$ and then $d_{G\wr H}(u)$ is not minimal, then one among $\{d_H(y_i)\}_{i=1,\ldots, n}$ or $d_G^p(x)$ cannot be minimal, and the statement follows.
\end{proof}

\begin{corollary}
If $H$ and $H'$ are two distance-balanced graphs of the same order, then $G\wr H$ is distance-balanced if and only if $G\wr H'$ is distance-balanced.
\end{corollary}

\begin{remark}
Another consequence of Theorem \ref{teom} is the equivalence of the TS-problems for $G$ with the classical problems for $G\wr H$, where $H$ is any distance-balanced graph. More precisely, let $H=(V_H,E_H)$ be a distance-balanced graph of order $m$, and $p=\frac{m-1}{m}$. Then we have:
\begin{equation*}
\begin{split}
\begin{split} &x\in V_G\\ \mbox{ is solution} &\mbox{ of  Problem \ref{tp1}}\end{split} \quad &\iff\quad  \begin{split}(y_1,\ldots&,y_n)x\in V_{G\wr H} \\ \mbox{ is solution} &\mbox{ of Problem \ref{p1}}\end{split}
\\\\
\begin{split}
&G \\
\mbox{is solution } &\mbox{of Problem \ref{tp2}}
\end{split}
\quad &\iff \quad
\begin{split}
&G\wr H \\
\mbox{is solution} &\mbox{ of Problem \ref{p2}}
\end{split}
\end{split}
\end{equation*}
\end{remark}

\begin{example}
We know from Example \ref{ewheel} that the graph $W_7$ is {\scriptsize $\frac{1}{2}$}TS-distance-balanced. By virtue of Theorem \ref{teom}, the graph $W_7\wr K_2$ is distance-balanced. Moreover, the graph $W_7\wr K_2$ has order $896$, it is non-regular (since $W_7$ is non-regular), and it is not bipartite (since $W_7$ is not bipartite). On the other hand, a direct computation shows that the cardinality of $W_{uv}^{W_7\wr K_2}$ is not constant on the set of edges, so that the graph $W_7\wr K_2$ is not \emph{nicely distance-balanced} (see \cite{nice}).
\end{example}

\begin{example}
We know from Example \ref{epiccolo} that the distance-balanced graph $H_9$ is not {\scriptsize $\frac{1}{2}$}TS-distance-balanced. As a consequence,
the graph $H_9\wr K_2$ is not distance-balanced.
\end{example}
\begin{corollary}
$G\wr H$ distance-balanced $\notimplies$ $G$ distance-balanced;\\
$G,H$ distance-balanced $\notimplies$  $G\wr H$ distance-balanced.
\end{corollary}

We conclude this section by investigating the class of graphs $G$ such that $G\wr H$ is distance-balanced whenever $H$ is distance-balanced. By virtue of Theorem \ref{teom}, this class must contain the class of TS-distance-balanced graphs. The two classes actually coincide, as we will prove in Theorem \ref{fine}. We need a preliminary definition and lemma.

\begin{defi}\label{tdv}
The \emph{total distance vector of} the vertex $u\in V_G$ is the $(n+1)$-vector
$$
W_{\rho}(u,G)=(W_{\rho_0}(u,G), W_{\rho_1}(u,G),\ldots, W_{\rho_n}(u,G)),
$$
where, for each $k\in \{0,1,\ldots,n\}$, we set
$$
W_{\rho_k}(u,G):=\sum_{A\subseteq V_G, \, |A|=k} \sum_{v\in V_G} \rho_A(u,v).
$$
\end{defi}
\noindent In particular, observe that $W_{\rho_0}(u,G)$ is the non-normalized total distance of the vertex $u$ in $G$.
\begin{lemma}\label{dw}
For every $u\in V_G$ and for every $p\in[0,1]$, we have:
$$
d_G^p(u)=  \frac{1}{n}\sum_{k=0}^n  p^k (1-p)^{n-k} W_{\rho_k}(u,G).
$$
\end{lemma}
\begin{proof}
The claim follows by combining Equation \eqref{dp} with Definition \ref{tdv}, since the expression of $p_A$ in Equation \eqref{pA} only depends on the cardinality of $A$.
\end{proof}

\begin{theorem}\label{fine}
Let $G=(V_G,E_G)$ be a graph of order $n$. The following are equivalent.
\begin{enumerate}
\item $G$ is TS-distance-balanced;
\item $G\wr H$ is distance-balanced for every distance-balanced graph $H$;
\item $G\wr K_{n^3 2^n}$ is distance-balanced;
\item $G$ is $p$TS-distance-balanced for more than $n$ distinct values of $p\in[0,1];$
\item the total distance vector  $W_{\rho}(u,G)$ does not depend on the particular vertex $u\in V_G$.
\end{enumerate}
\end{theorem}
\begin{proof}${}$

(1)$\implies$(2) It is a consequence of Theorem \ref{teom}.

(2)$\implies$(4) If $G\wr H$ is distance-balanced for every distance-balanced graph $H$, in particular $G\wr K_m$ is distance-balanced for $m=2,\ldots, n+2$, and then, by Theorem \ref{teom}, the graph $G$ is $p$TS-distance-balanced for each $p\in \left\{\frac{1}{2},\frac{2}{3},\ldots, \frac{n+1}{n+2}\right\}$.

(4)$\implies$(5) For a given vertex $u\in V_G$, we define the following polynomial of degree $n$ in the variable $x$
$$
P_u(x):=\sum^n_{k=0} x^{k}(1-x)^{n-k} W_{\rho_k}(u,G).
$$
By Lemma \ref{dw}, we have $P_u(p)=n d_G^p(u)$. Combining with hypothesis (4), for any pair $u,v\in V_G$, the polynomials $P_u$ and $P_v$ share more than $n$ evaluations, and so $P_u=P_v$. It is an easy exercise to prove that this implies that $W_{\rho_k}(u,G)=W_{\rho_k}(v,G)$, for each $k=0,1,\ldots,n$, and so $W_\rho(u,G)=W_\rho(v,G)$.

(5)$\implies$(1) It is a consequence of Lemma \ref{dw} and Theorem \ref{Tmedia}.

(2)$\implies$(3) It is true since the graph $K_{n^3 2^n}$ is distance-balanced.

(3)$\implies$(5) As we already observed in Remark \ref{obs}, for every $A\subseteq V_G$, for every $u,v\in V_G$ we have $\rho_A(u,v)<n^2$. Since the number of subsets of $V_G$ having cardinality $k$ is clearly less than $2^n$, this implies
\begin{equation}\label{bound}
0<W_{\rho_k}(u,G)=\sum_{A\subseteq V_G, \, |A|=k} \sum_{v\in V_G} \rho_A(u,v)<2^n n^3.
\end{equation}
We set $m:=2^n n^3$ and $p:=\frac{m-1}{m}$. Since, by hypothesis, $G\wr K_{m}$ is distance-balanced, it follows that, for every $u,v\in V_G$:
$$
d^p_G(u)nm^n=d^p_G(v)nm^n,
$$
and then, by Lemma \ref{dw}:
\begin{equation}\label{ff}
\sum^n_{k=0} (m-1)^k W_{\rho_k}(u,G)=\sum^n_{k=0} (m-1)^k W_{\rho_k}(v,G).
\end{equation}
By virtue of Equation \eqref{bound} we can regard the quantities $W_{\rho_k}(u,G)$ (resp.\  $W_{\rho_k}(v,G)$) as the digits of $d^p_G(u)nm^n$ (resp.\ $d^p_G(v)nm^n$) in base $(m-1)$; therefore, Equation \eqref{ff} implies that $W_{\rho_k}(u,G)=W_{\rho_k}(v,G)$, for each $k=0,1,\ldots,n$, and so $W_\rho(u,G)=W_\rho(v,G)$.
\end{proof}

\begin{example}\label{explicitvector}
Lemma \ref{dw} and the characterization (5) in Theorem \ref{fine} make us able to investigate distance-balancedness  (at least in those cases for which the total distance vector is known) simply by studying roots of  polynomials. For the graph $H_9$ from Example \ref{epiccolo}, by using the total distance vectors
$$
W_\rho(v_1,H_9) =(14,252,1345,3711,6279,6941,5065,2363,641,77)
$$
$$
W_\rho(v_2,H_9) =(14,252,1360,3762,6333,6933,5001,2307,620,74),
$$
we have verified that $H_9$ is $p$TS-distance-balanced if and only if $p=0$ or $p\approx 0.9312$, which is the unique real root of the polynomial $3x^5+15x^4+23x^3+3x^2-21x-15$.
\end{example}

As we already observed in Remark \ref{obs2}, vertex-transitive graphs satisfy the equivalent properties of Theorem \ref{fine}. Moreover, we recall that  regularity, vertex-transitivity and bipartiteness are all properties preserved by the wreath product. This yields the following infinite families of examples.

\begin{example}\label{famiglie}
Let $H_9$ be the non-regular non-bipartite distance-balanced graph from Example \ref{epiccolo}. Let $H_{24}$ be the Handa graph which is non-regular, bipartite and distance-balanced  \cite{Handa}. Consider the \emph{Generalized Petersen Graph} $P(7,3)$, that is regular, distance-balanced but not vertex-transitive (see \cite{distancebalancedintro}). Then:
\begin{itemize}
\item $\{K_n \wr H_9\}_{n\in \mathbb N}$ is a family of non-regular, non-bipartite, distance-balanced graphs;
\item $\{K_n \wr P(7,3)\}_{n\in \mathbb N}$ is a family of regular, non-vertex transitive, distance-balanced graphs;
\item $\{K_{n,n} \wr H_{24}\}_{n\in \mathbb N}$ is a family of non-regular, bipartite, distance-balanced graphs.
\end{itemize}
It is clear that, in order to obtain other infinite families with the same properties, one can replace $K_{n,n}$ or $K_n$ with any (bipartite or not) vertex-transitive graph, and the second factor with any distance-balanced graph sharing the same properties of regularity, vertex-transitivity, bipartiteness.
\end{example}

\section{Conclusions}

Vertex-transitive graphs are TS-distance-balanced. More generally, if $u$ and $v$ are vertices of a graph $G=(V_G,E_G)$ for which there exists $\varphi\in \operatorname{Aut}(G)$ such that $\varphi(u)=v$, one has $W_{\rho}(u,G)=W_{\rho}(v,G)$. In other words, the total distance vector is constant on the orbits of $V_G$ under the action of $\operatorname{Aut}(G)$. This suggests that it is possible to use it as an \emph{invariant} in order to distinguish vertices: it is a finer invariant than the standard total distance (see Example \ref{explicitvector}). Therefore, it is natural to ask whether this invariant is \emph{complete} on the orbit partition of vertices. The following question is a total-distance-analogue of Question 1 in \cite{cavadonno} about \emph{the Wiener vector} and the \emph{isomorphism problem}.
\begin{question}\label{qq}
Does there exist a graph $G=(V_G,E_G)$ with two vertices $u,v\in V_G$ for which there exists no automorphism $\varphi$ such that $\varphi(u)=v$, but $W_{\rho}(u,G)=W_{\rho}(v,G)$?
\end{question}
\noindent A negative answer would imply that the equivalent conditions of Theorem \ref{fine} are also equivalent to the vertex-transitivity property.

A last remark is that, regardless of the answer to our Question \ref{qq}, the wreath product construction  produces new infinite families  of distance-balanced graphs, which cannot be obtained via the classical graph products. Moreover, the graphs in these families possibly inherit good properties from their factors (see Example \ref{famiglie}). We believe that this new approach may provide different examples and counterexamples in the field of distance-balancedness and its generalizations, and for this reason it deserves to be further investigated and exploited.

\end{document}